\documentclass[a4paper, 10pt]{amsart}
\usepackage{amsmath}
\usepackage{amssymb}
\usepackage[T1]{fontenc}
\usepackage{enumerate}

\def\codim{\operatorname{codim}}

\def\Gr{\operatorname{Gr}}

\def\im{\operatorname{Im}}
\def\re{\operatorname{Re}}

\def\CC{\mathbb{C}} %C
\def\DD{\mathbb{D}} %D
\def\BB{\mathbb{B}} %B
\def\RR{\mathbb{R}} %R
\def\TT{\mathbb{T}} %T (brzeg D)
\def\CDD{\overline{\DD}} %domkniecie D
\def\OO{\mathcal{O}} %holomorficzne
\def\RE{\textnormal{Re}\,} %czesc rzeczywista
\def\IM{\textnormal{Im}\,} %czesc urojona
\def\LEBT{\mathcal{L}^{\TT}} %znaczek L^T, dla miary Lebesgue'a na okregu jednostkowym

\def\PDER#1#2{\frac{\partial #1}{\partial #2}} %pochodna czastkowa funkcji #1 wzgleem zmiennej #2
\def\DDER#1#2{\frac{d #1}{d #2}} %pochodna funkcji jednej zmiennej #1 wzgleem zmiennej #2

\def\HH{\mathcal{H}} %rodzina H
\def\MMM{\mathcal{M}} %rodzina M

\def\BLHL#1{\bar\lambda h_{#1}(\lambda)} %\bar\lambda h_j(\lambda)
\def\BLHSL#1{\bar\lambda h_{#1}^*(\lambda)} %\bar\lambda h_j^*(\lambda)
\def\LBHSL#1{\lambda\overline{h_{#1}^*(\lambda)}} %\bar\lambda h_j^*(\lambda)

\newtheorem{cor}{Corollary}[section]
\newtheorem{lem}[cor]{Lemma}
\newtheorem{obs}[cor]{Observation}
\newtheorem{prop}[cor]{Proposition}
\newtheorem{thm}[cor]{Theorem}

\theoremstyle{remark}
\newtheorem{rem}[cor]{Remark}

\theoremstyle{definition}

\newtheorem{ex}[cor]{Example}

\title[Complex geodesics and $\mathbb C$-convexity]{Complex geodesics in convex domains and $\mathbb C$-convexity of semitube domains}

\author{Sylwester Zaj\k{a}c}
\address{Institute of Mathematics of the Polish Academy of Sciences, \'{S}niadeckich 8, 00-656 Warszawa, Poland}
\email{sylwester.a.zajac@gmail.com}

\author{Pawe\l{} Zapa\l owski}
\address{Faculty of Mathematics and Computer Science, Jagiellonian University, \L o\-ja\-sie\-wi\-cza 6, 30-348 Krak\'ow, Poland}
\email{Pawel.Zapalowski@im.uj.edu.pl}

\thanks{The authors are partially supported by the Polish National Science Center (NCN) grant UMO-2014/15/D/ST1/01972}


\begin{document}

\begin{abstract}In the paper the complex geodesics of a convex domain in $\CC^n$ are studied. One of the main results of the paper provides certain necessary condition for a holomorphic map to be a complex geodesic for a convex domain in $\CC^n$. The established condition is of geometric nature and it allows to find a formula for every complex geodesic. The $\mathbb C$-convexity of semitube domains is also discussed.
\end{abstract}

\maketitle

\section{Introduction}

The aim of the paper is twofold. First, to provide certain condition which allows to find formulas for all complex geodesics in an arbitrary convex domain in $\CC^n$ (it is the content of Section~\ref{sect:cgicd}) and second, to discuss $\mathbb C$-convexity of semitube domains in $\CC^n$ (see Section~\ref{sect:ccstd}).

\subsection{Complex geodesics} A holomorphic map $\varphi:\DD\longrightarrow D$ is called a \emph{complex geodesic} for a domain $D\subset\CC^n$, if it admits a \emph{left inverse}, that is, a holomorphic function $f:D\longrightarrow\DD$ such that $f\circ\varphi$ is the identity of $\DD$ (for the notation and terminology we refer the reader to the beginning of Section~\ref{sect:p}). These maps, being fundamental objects of research in complex analysis, are precisely the holomorphic isometries between the unit disc $\DD\subset\CC$ equipped with the Poincar\'{e} distance and the domain $D$ equipped with the Carath\'{e}odory pseudodistance (see \cite{vesentini197939}, \cite{vesentini1981375}, \cite{vesentini1982211}). They are inseparably connected with the famous Lempert theorem, guaranteeing that if $D$ is convex, then through an arbitrarily chosen pair of points of $D$ one can pass a complex geodesic (see \cite{lempert1981427} or \cite[Chapter 8]{jarnicki1993} and also \cite{vigue1985345}, \cite{royden1983}).

For an integer $0\leq d\leq n$ denote by $\mathcal{A}_d^n$ the set of all convex domains $D\subset\CC^n$ such that $D$ contains no complex affine lines and a maximal real affine subspace contained in $D$ is of the form $z_0+\lbrace 0\rbrace^{n-d}\times(i\RR)^d$ for some $z_0\in D$. It is clear that every convex domain in $\CC^n$ is affinely equivalent to a Cartesian product of some $\CC^k$ and a convex domain containing no complex affine lines. Therefore, in our investigations of complex geodesics we can restrict to the latter type of domains. Importantly, each of them can be transformed, by a complex affine isomorphism, to an element of $\mathcal{A}_d^n$. In Lemma~\ref{lem_D_fi_miary_graniczne} we will see that if $D\in\mathcal{A}_d^n$ and $\varphi\in\OO(\DD,D)$, then $\varphi$ admits the \emph{boundary measure} (see Section~\ref{sect:p}) of a special form. This justifies introducing the family $\mathcal{A}_d^n$, as well as the fact that it becomes the area of considered domains in majority of our investigations.

Theorem~\ref{th_warunek_konieczny_na_geodezyjna_w_Adn}, which is the main result of Section~\ref{sect:cgicd}, presents certain necessary condition for a holomorphic map $\varphi:\DD\longrightarrow D$ to be a complex geodesic for a domain $D\in\mathcal{A}_d^n$. The condition obtained is of geometric nature. It describes the absolutely continuous part of $\varphi$'s boundary measure in its Lebesgue-Radon-Nikodym decomposition with respect to the Lebesgue measure $\LEBT$ on $\TT$. As for the singular part, Theorem~\ref{th_warunek_konieczny_na_geodezyjna_w_Adn} gives a restraint on it and demonstrates that it can hardly be strengthened. In the latter part of Section~\ref{sect:cgicd} we also prove Theorem~\ref{th_warunek_wystarczajacy_na_geodezyjna_w_Adn}, being an inverse, although not in full extent, of Theorem~\ref{th_warunek_konieczny_na_geodezyjna_w_Adn}.

In the paper we extend the methods from \cite{zajac20151337} and \cite{zajac20161865}, applied there to establish a complete description of all complex geodesics in convex tube domains (that is, precisely those from the family $\mathcal{A}_n^n$). For a domain $D\in\mathcal{A}_d^n$ one can say that there is a kind of 'tube part' of $D$ at the last $d$ coordinates. And in fact, we employ some key argumentations from the aforementioned publications mainly to deal with the last $d$ coordinates of a complex geodesic.

\subsection{$\CC$-convexity}As it was already mentioned, Section~\ref{sect:ccstd} is devoted to the study the notion of $\mathbb C$-convexity in the class of the so-called semitube domains, which we define as follows. Let $\Pi:\mathbb C^n\longrightarrow\mathbb R^{2n-1}$ be defined by
\begin{equation*}
\Pi(z_1,\dots,z_n):=(\re z_1,\im z_1,\dots,\re z_{n-1},\im z_{n-1},\re z_n).
\end{equation*}
The \emph{semitube domain} (\emph{set}) \emph{with the base} $B$ being a domain (set) lying in $\mathbb R^{2n-1}$ ($n>1$) is defined as follows
$$
\mathcal S_B:=\Pi^{-1}(B).
$$
It is a generalization of semitube domains (sets) in $\mathbb C^2$ introduced in \cite{burgues2012685} and studied in \cite{kosinski2015241}. Note that any tube domain (set) is a semitube one. Moreover, any domain of the family $\mathcal A_d^n$ with $d\geq1$ is a semitube domain.

Recall that a domain $D\subset\mathbb C^n$ is called (cf.~\cite{hormander1994}, \cite{andersson2004}) $\mathbb C$-\emph{convex}, if for any affine complex line $L$ such that $L\cap D\neq\varnothing$, the set $L\cap D$ is connected and simply connected. Observe that any convex domain is $\mathbb C$-convex, but the converse does not hold in general.

The notion of $\mathbb C$-convexity plays an important role in geometric function theory. It is a consequence of the celebrated Lempert theorem (cf.~\cite{lempert1981427}) that the property
\begin{equation}\label{eq:lempert}
\textit{the Lempert function and the Carath\'eodory distance of }D\textit{ coincide}
\end{equation}
holds for any bounded $\CC$-convex domain $D$ with $\mathcal C^2$-smooth boundary (cf.~\cite{jacquet2006303}). Any convex domain satisfies (\ref{eq:lempert}) too, since it can be exhausted by smooth bounded convex domains. It is an open problem, whether the property (\ref{eq:lempert}) holds for any bounded $\CC$-convex domain (cf.~Problem 4' in \cite{znamenskij2001123}). The first non-trivial example which supports this conjecture is the symmetrized bidisc. It is a bounded pseudoconvex domain with non-smooth boundary and with the property (\ref{eq:lempert}) (see \cite{costara2004d}) which cannot be exhausted by domains biholomorphic to convex domains (see \cite{edigarian2004189}) and which is $\mathbb C$-convex (\cite{nikolov2008149}). The second (and, up to now, the last one) non-trivial example sharing the above mentioned properties is the tetrablock (cf.~\cite{abouhajar2007717}, \cite{edigarian20131818}, and \cite{zwonek2013159}). To sum up, all known bounded domains with the property (\ref{eq:lempert}) which cannot be exhausted by domains biholomorphic to convex domains turn out to be $\CC$-convex! Thus the $\CC$-convexity seems to be a natural environment for the property (\ref{eq:lempert}) and hence becomes worth studying.

The main result of Section~\ref{sect:ccstd} is Theorem~\ref{thm:ccstd} which shows that the notions of the convexity and the $\CC$-convexity coincide in a large class of semitube domains. To be more specific, we give a simple geometric sufficient condition for a base of the semitube domain, which makes the notions of convexity and $\CC$-convexity of a semitube domain equal. Remark~\ref{rem:example}~(a) shows, that there are $\CC$-convex semitube domains which are not convex. It would be desirable to find the necessary condition for the base which makes equivalence between the convexity and the $\CC$-convexity of a semitube domain. Unfortunately, we have not been able to do this.

\section{Preliminaries}\label{sect:p}

Here is some notation. Throughout the paper $\mathbb D$ denotes the unit disc in the complex plane, by $\mathbb T$ we shall denote the unit circle, $\LEBT$ is the Lebesgue measure on $\TT$, $\|\cdot\|$ is the Euclidean norm in $\CC^n$ and $\mathbb B(a,r)$ stands for the Euclidean ball in $\mathbb C^n$ with center at $a$ and radius $r$. Additionally, by $\BB_{\RR^n}$ we denote the Euclidean unit ball in $\RR^n$ with center at the origin. For $z=(z_1,\dots,z_n)$ and $w=(w_1,\dots,w_n)\in\mathbb C^n$ let $z\bullet w:=\sum_{j=1}^nz_jw_j$ denote the standard dot product in $\mathbb C^n$ and $z=(z',z_n)\in\mathbb C^{n-1}\times\mathbb C$. By $\{e_1,\ldots,e_n\}$ we denote the canonical basis of $\mathbb C^n$ or $\RR^n$. For $A\subset\mathbb C^n$ we shall write $A_*:=A\setminus\{0\}$ and $\partial A$ to denote the boundary of the set $A$. Given two domains $D\subset\CC^n$ and $G\subset\CC^m$ by $\OO(D,G)$ we denote the space of all holomorphic mappings $D\longrightarrow G$. To shorten the notation we often write $f_{j,\ldots,k}:=(f_j,f_{j+1},\ldots,f_k)$ for a tuple $f=(f_1,\ldots,f_n)$ of objects and numbers $1\leq j\leq k\leq n$. Let us also note that the symbol $\bullet$ will be used, in a standard meaning, with measures and functions, e.g.~if $f=(f_1,\ldots,f_n)$ is a tuple of functions and $\mu=(\mu_1,\ldots,\mu_n)$ is a tuple of complex measures, then $f\bullet d\mu$ is the measure $f_1 d\mu_1+\ldots+f_n d\mu_n$, etc. Finally, by $H^1(\DD,\CC^n)$ we denote the family of all holomorphic maps $\DD\longrightarrow\CC^n$ with the components lying in the classical Hardy space $H^1$.

\smallskip
The next lemma, providing a kind of decomposition of $n$-tuples of real measures, plays a crucial role in the investigations made in Section~\ref{sect:cgicd}.

\begin{lem}[\cite{zajac20161865}, Lemma 2.1]\label{lem_rozklad_miary_varrho_d_nu}Let $\mu$ be an $n$-tuple of real Borel measures on $\TT$. Then there exist a unique finite positive Borel measure $\nu$ on $\TT$ singular to $\LEBT$, a unique, up to a set of $\nu$ measure zero, Borel-measurable map $\varrho:\TT\longrightarrow\partial\BB_{\RR^n}$ and a unique, up to a set of $\LEBT$ measure zero, Borel-measurable map $g:\TT\longrightarrow\RR^n$ with components in $L^1(\TT,\LEBT)$ such that
\begin{equation}\label{eq_lem_rozklad_miary_varrho_d_nu_postac_miary}
\mu=g\,d\LEBT+\varrho\,d\nu.
\end{equation}
In particular, $g\,d\LEBT$ and $\varrho\,d\nu$ are, respectively, the absolutely continuous part and the singular part of $\mu$ in its Lebesgue-Radon-Nikodym decomposition with respect to $\LEBT$.
\end{lem}

An $n$-tuple $\mu=(\mu_1,\ldots,\mu_n)$ of real Borel measures on $\TT$ is called the \emph{boundary measure} of a map $\varphi\in\OO(\DD,\CC^n)$, if there holds the Schwarz formula
\begin{equation}\label{eq_g_schwarz_formula}
\varphi(\lambda)=\frac{1}{2\pi}\int_{\TT}\frac{\zeta+\lambda}{\zeta-\lambda}d\mu(\zeta)+i\IM\varphi(0),\quad\lambda\in\DD,
\end{equation}
or, equivalently, the Poisson formula
\begin{equation}\label{eq_g_poisson_formula}
\RE\varphi(\lambda)=\frac{1}{2\pi}\int_{\TT}\frac{1-|\lambda|^2}{|\zeta-\lambda|^2}d\mu(\zeta),\quad\lambda\in\DD.
\end{equation}
Here the integration is meant coordinate-wise. Define
$$
\MMM^n:=\lbrace\varphi\in\OO(\DD,\CC^n):\varphi\text{ admits the boundary measure}\rbrace.
$$
The family $\MMM^1$ contains, among others, every holomorphic function with non-negative or non-positive real part (see e.g.~\cite[p.~5]{koosis1998}). The correspondence between elements of $\MMM^n$ and their boundary measures is one-to-one, up to adding an imaginary constant. If a mapping $\varphi\in\MMM^n$ has the boundary measure $\mu$ and $\mu=g\,d\LEBT+\varrho\,d\nu$ is the decomposition introduced in Lemma~\ref{lem_rozklad_miary_varrho_d_nu}, then from the Fatou theorem (see \cite[p.~11]{koosis1998}) it follows that $\RE\varphi^*(\lambda)=g(\lambda)$ for $\LEBT$-a.e.~$\lambda\in\TT$. In consequence, the components of $\RE\varphi^*$ belong to $L^1(\TT,\LEBT)$ and $\mu=\RE\varphi^*\,d\LEBT+\varrho\,d\nu$. In the paper we will need the holomorphic maps induced by the parts of the decomposition of $\mu$, so let us set, for $\lambda\in\DD$,
\begin{align*}
\varphi^a(\lambda)&:=\frac{1}{2\pi}\int_{\TT}\frac{\zeta+\lambda}{\zeta-\lambda}\RE\varphi^*(\zeta)\,d\LEBT(\zeta)+i\IM\varphi(0),\\
\varphi^s(\lambda)&:=\frac{1}{2\pi}\int_{\TT}\frac{\zeta+\lambda}{\zeta-\lambda}\varrho(\zeta)\,d\nu(\zeta).
\end{align*}
Clearly $\varphi^a,\varphi^s\in\MMM^n$, $\IM\varphi^s(0)=0$ and $\varphi=\varphi^a+\varphi^s$. Moreover, employing the Fatou theorem once again we deduce that $\RE\varphi^*(\lambda)=\RE(\varphi^a)^*(\lambda)$ and $\RE(\varphi^s)^*(\lambda)=0$ for $\LEBT$-a.e.~$\lambda\in\TT$. Finally, it is worthy to note that if $\varphi\in H^1(\DD,\CC^n)$, then $\varphi\equiv\varphi^a$.

\begin{obs}\label{obs_f_of_class_H1_if_f_and_if_has_b_measures_ogolniejsza}If $a,b\in\CC$ are linearly independent over $\RR$, $f\in\OO(\DD,\CC)$ and both $af$ and $bf$ admit the boundary measures, then $f\in H^1(\DD,\CC)$.
\end{obs}

\begin{proof} We may assume that $f(0)=0$. Let $\mu$ and $\nu$ be the boundary measures of $af$ and $bf$, respectively. One has that
$$
\frac{b}{2\pi}\int_{\TT}\frac{\zeta+\lambda}{\zeta-\lambda}d\mu(\zeta)=\frac{a}{2\pi}\int_{\TT}\frac{\zeta+\lambda}{\zeta-\lambda}d\nu(\zeta),\quad\lambda\in\DD.
$$
Writing the Taylor series' expansions at the origin of both sides, we get
$$
\sum_{n=1}^\infty \frac{b\lambda^n}{\pi}\int_{\TT}\bar\zeta^n\,d\mu(\zeta)=\sum_{n=1}^\infty \frac{a\lambda^n}{\pi}\int_{\TT}\bar\zeta^n\,d\nu(\zeta),\quad\lambda\in\DD.
$$
This yields that
$$
\int_{\TT}\bar\zeta^n\,d(b\mu-a\nu)(\zeta)=0,\quad n=1,2,\ldots.
$$
The theorem of the brothers Riesz yields that $b\mu-a\nu=g\,d\LEBT$ for a function $g\in L^1(\TT,\LEBT)$. Since $c:=i\IM(a\bar b)\neq 0$, from \eqref{eq_g_poisson_formula} applied to $af$ and $bf$ there follows the equality
$$
f(\lambda)=c^{-1}\left(\bar b\,\RE(af)-\bar a\,\RE(bf)\right)=\frac{1}{2\pi c}\int_{\TT}\frac{1-|\lambda|^2}{|\zeta-\lambda|^2}\,\bar g(\zeta)\,d\LEBT(\zeta),\quad\lambda\in\DD,
$$
which in turn leads to the desired conclusion, in view of \cite[p.~7]{koosis1998}.
\end{proof}

\section{Complex geodesics in convex domains}\label{sect:cgicd}

Given mappings $\varphi,h\in\OO(\DD,\CC^n)$ and a point $z\in\CC^n$ define the function $\psi_z\in\OO(\DD,\CC)$ by
\begin{equation}\label{eq_g_formula_re_psi_z_1}
\psi_z(\lambda):=\tfrac{\varphi(0)-\varphi(\lambda)}{\lambda}\bullet h(\lambda)+\tfrac{h(\lambda)-h(0)}{\lambda}\bullet(z-\varphi(0))+\lambda\,\overline{h(0)\bullet(z-\varphi(0))}.
\end{equation}
It can be also written as
\begin{equation}\label{eq_g_formula_re_psi_z_2}
\psi_z(\lambda)=\frac{h(\lambda)\bullet(z-\varphi(\lambda))-h(0)\bullet(z-\varphi(0))}{\lambda} + \lambda\,\overline{h(0)\bullet(z-\varphi(0))}.
\end{equation}
If for a $\lambda\in\TT$ the radial limits of $h$ and $\varphi$ exist at $\lambda$, then for each $z\in\CC^n$ the radial limit $\psi_z^*(\lambda)$ also exists and it holds that
\begin{equation}\label{eq_g_granice_radialne_re_psi_z}
\RE\psi_z^*(\lambda)=\RE\left[\BLHSL{}\bullet(z-\varphi^*(\lambda))\right],\quad z\in\CC^n.
\end{equation}
In most situations it will be clear for which maps $\varphi$ and $h$ the function $\psi_z$ is regarded. Otherwise, we shall write it with additional upper indexes, namely $\psi_z^{\varphi}$ or $\psi_z^{\varphi,h}$.

Let us recall a lemma which provides a sufficient condition for a holomorphic map to be a complex geodesic:

\begin{lem}[\cite{zajac20151337}, Lemma 3.4]\label{lem_g_ogolne_warunki_wystarczajace}Let $D\subset\CC^n$ be a domain and let $\varphi\in\OO(\DD,D)$. If there exists a map $h\in\OO(\DD,\CC^n)$ such that $\RE\psi_{\varphi(0)}(0)\neq 0$ and
$$
\RE\psi_z(\lambda)\leq 0,\quad\lambda\in\DD,\ z\in D,
$$
then $\varphi$ admits a left inverse in $D$.
\end{lem}

The next lemma ensures that for convex domains the condition from Lemma~\ref{lem_g_ogolne_warunki_wystarczajace} is in fact an equivalent one.

\begin{lem}\label{lem_g_ogolne_warunki_konieczne}Let $D\subset\CC^n$ be a convex domain and let $\varphi\in\OO(\DD,D)$ be a complex geodesic for $D$ with a left inverse $f\in\OO(D,\DD)$. Set
$$
h(\lambda):=\left(\PDER{f}{z_1}(\varphi(\lambda)),\ldots,\PDER{f}{z_n}(\varphi(\lambda))\right),\quad\lambda\in\DD.
$$
Then $\RE\psi_{\varphi(0)}(0)\neq 0$ and
$$
\RE\psi_z(\lambda)\leq 0,\quad\lambda\in\DD,\ z\in D.
$$
\end{lem}

\begin{proof}
Differentiating both sides of the equality $f(\varphi(\lambda))=\lambda$ we get
\begin{equation}\label{eq_g_tubowe_lowk_1145}
h(\lambda)\bullet\varphi'(\lambda)=1,\quad\lambda\in\DD.
\end{equation}
In particular,
$$
\RE\psi_{\varphi(0)}(0) = -\RE\left[h(0)\bullet\varphi'(0)\right]\neq 0.
$$

For $z\in D$ and $t\in[0,1]$ set
$$
f_{z,t}(\lambda):=f((1-t)\varphi(\lambda)+tz),\quad\lambda\in\DD.
$$
Clearly $f_{z,t}\in\OO(\DD,\DD)$ and $f_{z,0}(\lambda)=\lambda$. One can check that
\begin{equation}\label{eq_g_tubowe_lowk_123488}
\left.\DDER{|f_{z,t}(\lambda)|^2}{t}\right|_{t=0}=2\,\RE\left[\BLHL{}\bullet(z-\varphi(\lambda))\right],\quad\lambda\in\DD,\ z\in D.
\end{equation}
On the other hand, in view of \cite[Lemma 1.2.4]{abate1989} we have
$$
|f_{z,t}(\lambda)|-|\lambda|\leq\frac{2|f_{z,t}(0)|}{1+|f_{z,t}(0)|}(1-|\lambda|),\quad\lambda\in\DD,\ z\in D.
$$
Thus
$$
\frac{|f_{z,t}(\lambda)|^2-|f_{z,0}(\lambda)|^2}{t}=\frac{|f_{z,t}(\lambda)|^2-|\lambda|^2}{t} \leq 2\,\frac{|f_{z,t}(\lambda)|-|\lambda|}{t} \leq \frac{4|\frac1t f_{z,t}(0)|}{1+|f_{z,t}(0)|}(1-|\lambda|).
$$
Taking limit for $t$ tending to $0$ we get
\begin{equation}\label{eq_g_tubowe_lowk_123499}
\left.\DDER{|f_{z,t}(\lambda)|^2}{t}\right|_{t=0} \leq 4(1-|\lambda|)\left|\left.\DDER{f_{z,t}(0)}{t}\right|_{t=0}\right| \leq 4(1-|\lambda|)\,|h(0)\bullet(z-\varphi(0))|.
\end{equation}
Now from \eqref{eq_g_tubowe_lowk_123488} and \eqref{eq_g_tubowe_lowk_123499} we conclude that
\begin{equation*}
\RE\left[\BLHL{}\bullet(z-\varphi(\lambda))\right] \leq 2(1-|\lambda|)\,|h(0)\bullet(z-\varphi(0))|,\quad\lambda\in\DD,\ z\in D.
\end{equation*}
Dividing this inequality by $|\lambda|^2$ we obtain
\begin{equation}\label{eq_g_tubowe_lowk_1234}
\RE\left[\frac{h(\lambda)\bullet(z-\varphi(\lambda))}{\lambda}\right] \leq 2\,\frac{1-|\lambda|}{|\lambda|^2}\,|h(0)\bullet(z-\varphi(0))|,\quad\lambda\in\DD_*,\ z\in D.
\end{equation}

Fix $z\in D$. By \eqref{eq_g_formula_re_psi_z_2} and the above inequality, the function $\RE\psi_z$ is bounded from above on the set $\DD\setminus\frac12\DD$. The maximum principle yields that it is bounded from above on $\DD$. In particular, $\LEBT$-almost all of its radial limits exist and, in view of \eqref{eq_g_tubowe_lowk_1234}, for $\LEBT$-a.e.~$\lambda\in\TT$ one has that
$$
\RE\psi_z^*(\lambda)\leq 0.
$$
Thus, from the maximum principle it follows that $\RE\psi_z(\lambda)\leq 0$ for $\lambda\in\DD$.
\end{proof}

For a domain $D\in \mathcal A_d^n$ define
\begin{align*}
W_D&:=\lbrace v\in\CC^n:\sup_{z\in D}\RE(z\bullet v)<\infty\rbrace,\\
S_D&:=\lbrace y\in\RR^d:\forall\,v\in W_D:(0,y)\bullet v\leq 0\rbrace.
\end{align*}
The sets $W_D$ and $S_D$ are convex infinite cones and one can check that
\begin{equation}\label{eq_zt0y_in_D_if_z_in_D_0y_in_SD_t_geq_0}
z+(0,y)+(0,ix)\in D,\quad z\in D,\ y\in S_D,\ x\in\RR^d.
\end{equation}
Moreover, $W_D\subset\CC^{n-d}\times\RR^d$ and the interior of $W_D$ with respect to $\CC^{n-d}\times\RR^d$ is non-empty. The latter statement is justified as follows: assuming, to the contrary, that the aforementioned interior is empty, one can find a non-zero vector $v_0\in\CC^{n-d}\times\RR^d$ such that $\RE(v_0\bullet v)=0$ for all $v\in W_D$. Fix $z_0\in D$. By the choice of $d$ one has that $z_0+v_0\cdot\RR\not\subset D$, so $p_0:=z_0+t_0 v_0\in\partial D$ for some $t_0\in\RR\setminus\lbrace 0\rbrace$. Now, if $v\in\CC^n$ is chosen so that $\RE((z-p_0)\bullet v)<0$ for all $z\in D$, then $v\in W_D$ and $\RE(v_0\bullet v)\neq 0$. A contradiction.

\begin{lem}\label{lem_D_fi_miary_graniczne}Let $D\in\mathcal{A}_d^n$ and $\varphi=(\varphi_1,\ldots,\varphi_n)\in\OO(\DD,\CC^n)$. Then $\varphi(\DD)\subset\overline{D}$ if and only if the following three conditions hold:
\begin{enumerate}[(i)]
\item\label{lem_D_fi_miary_graniczne_fi_j_in_H1} $\varphi_1,\ldots,\varphi_{n-d}\in H^1(\DD,\CC)$, $\varphi_{n-d+1},\ldots,\varphi_n\in\MMM^1$,
\item\label{lem_D_fi_miary_graniczne_granice_radialne_w_domknieciu} $\varphi^*(\lambda)\in\overline{D}$ for $\LEBT$-a.e.~$\lambda\in\TT$,
\item\label{lem_D_fi_miary_graniczne_fi_j_has_boundary_measure} the boundary measure of $\varphi$ is of the form
    $$
    \RE\varphi^*\,d\LEBT+(0,\varrho)\,d\nu,
    $$
    where $\nu$ is a finite positive Borel measure on $\TT$, singular to $\LEBT$, and $\varrho:\TT\longrightarrow\partial\BB_{\RR^d}$ is a Borel-measurable map such that $\varrho(\lambda)\in S_D$ for $\nu$-a.e.~$\lambda\in\TT$.
\end{enumerate}
\end{lem}

From the above lemma it follows that if $D\in\mathcal{A}_d^n$, $\varphi\in\OO(\DD,\CC^n)$ and $\varphi(\DD)\subset\overline{D}$, then $\varphi_{1,\ldots,n-d}^a\equiv\varphi_{1,\ldots,n-d}$, $\varphi_{1,\ldots,n-d}^s\equiv 0$ and $\varphi^a(\DD)\subset\overline{D}$. Moreover, the following useful equalities are implied by \eqref{lem_D_fi_miary_graniczne_fi_j_in_H1}:
\begin{align}
\label{eq_poisson_formula_fi_a_1_nd} \varphi_{1,\ldots,n-d}^a(\lambda)&=\frac{1}{2\pi}\int_{\TT}\frac{1-|\lambda|^2}{|\zeta-\lambda|^2}\varphi_{1,\ldots,n-d}^*(\zeta)\,d\LEBT(\zeta),\\
\label{eq_poisson_formula_re_fi_a_nd1_n} \RE\varphi_{n-d+1,\ldots,n}^a(\lambda)&=\frac{1}{2\pi}\int_{\TT}\frac{1-|\lambda|^2}{|\zeta-\lambda|^2}\RE\varphi_{n-d+1,\ldots,n}^*(\zeta)\,d\LEBT(\zeta),
\end{align}
when $\lambda\in\DD$.

\begin{proof}Assume that $\varphi(\DD)\subset\overline{D}$ and take an Euclidean ball $B\subset W_D\subset\CC^{n-d}\times\RR^d$. For each $v\in B$ the real part of the function $\varphi(\cdot)\bullet v$ is bounded from above, so this function belongs to $\MMM^1$. Hence, also
\begin{equation}\label{eq_ldfmg_fi_bullet_vw_in_M}
\varphi(\cdot)\bullet (v-w)\in\MMM^1,\quad v,w\in B.
\end{equation}
Therefore, $\varphi$ admits the boundary measure, denoted below by $\mu=(\mu_1,\ldots,\mu_n)$. In particular, $\LEBT$-almost all radial limits of $\varphi$ exist, what implies \eqref{lem_D_fi_miary_graniczne_granice_radialne_w_domknieciu}. Employing \eqref{eq_ldfmg_fi_bullet_vw_in_M} once again and using Observation~\ref{obs_f_of_class_H1_if_f_and_if_has_b_measures_ogolniejsza} we obtain that the functions $\varphi_1,\ldots,\varphi_{n-d}$ are of the class $H^1$. This gives the condition \eqref{lem_D_fi_miary_graniczne_fi_j_in_H1} and the equality $\mu_{1,\ldots,n-d}=\RE\varphi_{1,\ldots,n-d}^*\,d\LEBT$. What is more, if $r\to 1^-$, then for $j=1,\ldots,n-d$ the measures $\RE\varphi_j(r\lambda)\,d\LEBT(\lambda)$ and $\IM\varphi_j(r\lambda)\,d\LEBT(\lambda)$ converge weakly-* to $\RE\varphi_j^*(\lambda)\,d\LEBT(\lambda)$ and $\IM\varphi_j^*(\lambda)\,d\LEBT(\lambda)$, respectively.

Write $\mu=\RE\varphi^*\,d\LEBT+\widetilde{\varrho}\,d\nu$, as in Lemma~\ref{lem_rozklad_miary_varrho_d_nu}. Since the measures $\mu_1,\ldots,\mu_{n-d}$ are absolutely continuous with respect to $\LEBT$, one has that $\widetilde{\varrho}(\lambda)\in\lbrace 0\rbrace^{n-d}\times\RR^d$ for $\nu$-a.e.~$\lambda\in\TT$. Thus, $\widetilde{\varrho}=(0,\varrho)$ for a Borel-measurable mapping $\varrho:\TT\longrightarrow\partial\BB_{\RR^d}$.

We claim that $\varrho(\lambda)\in S_D$ for $\nu$-a.e.~$\lambda\in\TT$. Take an arbitrary vector $v=(v_1,\ldots,v_n)\in W_D$ and a real constant $C$ so that $\RE(z\bullet v)\leq C$ for all $z\in D$. Since $v_{n-d+1},\ldots,v_n\in\RR$, for all $\lambda\in\TT$ and $r\in(0,1)$ we have
$$
\sum_{j=1}^{n-d}\left(\RE v_j\,\RE\varphi_j(r\lambda)-\IM v_j\,\IM\varphi_j(r\lambda)\right)+\sum_{j=n-d+1}^n v_j\RE\varphi_j(r\lambda)\leq C.
$$
Multiplying both sides by $\LEBT$ and taking the weak-* limits for $r\to 1^-$ we obtain
\begin{equation}\label{eq_ldfmg_sum_j_re_vj_re_fijs_im_vj_im_fijs}
\sum_{j=1}^{n-d}\left(\RE v_j\,\RE\varphi_j^*\,d\LEBT-\IM v_j\,\IM\varphi_j^*\,d\LEBT\right)+\sum_{j=n-d+1}^n v_j\,d\mu_j\leq C\,d\LEBT.
\end{equation}
There exists a Borel subset $S\subset\TT$ such that
$$
\LEBT(S)=0,\;\nu(\TT\setminus S)=0.
$$
Clearly
\begin{equation*}
\chi_S\,d\LEBT=0,\; \chi_S\,d\mu=(0,\varrho)\,d\nu.
\end{equation*}
Hence, multiplying both sides of \eqref{eq_ldfmg_sum_j_re_vj_re_fijs_im_vj_im_fijs} by $\chi_S$ we get that $(v\bullet(0,\varrho))\,d\nu\leq 0$, what leads to the conclusion that the inequality $v\bullet(0,\varrho)\leq 0$ is valid $\nu$-almost everywhere on $\TT$. This 'almost everywhere' may a priori depend on $v$, but one can omit this problem in the following way. Take a dense subset $\lbrace v_j:j=1,2,\ldots\rbrace\subset W_D$ and for each $j$ choose a Borel set $A_j\subset\TT$ so that $\nu(\TT\setminus A_j)=0$ and $v_j\bullet(0,\varrho(\lambda))\leq 0$ for every $\lambda\in A_j$. Denote $A:=\cap_{j=1}^\infty A_j$. It is now clear that $\nu(\TT\setminus A)=0$ and $v\bullet(0,\varrho(\lambda))\leq 0$ all $v\in W_D$ and $\lambda\in A$. Thus, $\varrho(\lambda)\in S_D\text{ for }\nu\text{-a.e. }\lambda\in\TT$. The statement \eqref{lem_D_fi_miary_graniczne_fi_j_has_boundary_measure} follows from this.

Now assume that $\varphi$ satisfy the conditions \eqref{lem_D_fi_miary_graniczne_fi_j_in_H1}, \eqref{lem_D_fi_miary_graniczne_granice_radialne_w_domknieciu} and \eqref{lem_D_fi_miary_graniczne_fi_j_has_boundary_measure}. Fix $\lambda\in\DD$. From the equalities \eqref{eq_poisson_formula_fi_a_1_nd} and \eqref{eq_poisson_formula_re_fi_a_nd1_n}, the fact that $(1-|\lambda|^2)|\zeta-\lambda|^{-2}\,d\LEBT(\zeta)$ is a probabilistic measure and the convexity of $D$ it follows that $\varphi^a(\lambda)\in\overline{D}$. Moreover, since $S_D$ is a closed convex infinite cone, we have
$$
\RE\varphi_{n-d+1,\ldots,n}^s(\lambda) = \frac{1}{2\pi}\int_{\TT}\frac{1-|\lambda|^2}{|\zeta-\lambda|^2}\varrho(\zeta)\,d\nu(\zeta)\in S_D.
$$
Thus $\varphi(\lambda)=\varphi^a(\lambda)+\varphi^s(\lambda)\in\overline{D}$, by \eqref{eq_zt0y_in_D_if_z_in_D_0y_in_SD_t_geq_0}.
\end{proof}

\begin{obs}\label{obs_h_of_class_H1_if_fi_geodesic}If $\varphi, h\in\OO(\DD,\CC^n)$ are such that $\RE\psi_z(\lambda)\leq 0$ for all $\lambda\in\DD$ and all $z$ from an Euclidean ball $B\subset\CC^n$, then $h\in H^1(\DD,\CC^n)$.
\end{obs}

\begin{proof}Define $\widetilde{h}(\lambda):=\lambda^{-1}(h(\lambda)-h(0))$. From the assumptions it follows that for each $z\in B$ the function $\psi_z$ admits the boundary measure.
Hence, so does the function
$$
\lambda\mapsto\widetilde{h}(\lambda)\bullet (z-w)=\psi_z(\lambda)-\psi_w(\lambda)+\lambda\,\overline{h(0)\bullet (z-w)}
$$
for each $z,w\in B$. In view of Observation~\ref{obs_f_of_class_H1_if_f_and_if_has_b_measures_ogolniejsza}, each coordinate of $\widetilde{h}$ is of the class $H^1$. This implies that $h$ lies in $H^1(\DD,\CC^n)$.
\end{proof}

For the sake of clarity, introduce the following family of mappings:
\begin{align*}
\HH^n   &:=\lbrace h\in\OO(\CC,\CC^n): \forall\lambda\in\TT:\,\BLHL{}\in\RR^n\rbrace,\\
\HH^n_+ &:=\lbrace h\in\OO(\CC,\CC^n): \forall\lambda\in\TT:\,\BLHL{}\in [0,\infty)^n\rbrace,\\
\HH^n_d &:=\lbrace h=(h_1,\ldots,h_n)\in H^1(\DD,\CC^n): h_{n-d+1},\ldots,h_n\in\HH^1\rbrace.
\end{align*}
It is elementary that if $h\in H^1(\DD,\CC^n)$ satisfies $\BLHSL{}\in\RR^n$ for $\LEBT$-a.e.~$\lambda\in\TT$, then $h\in\HH^n$ and $h(\lambda)=\bar a\lambda^2+b\lambda+a$ for some $a\in\CC^n$ and $b\in\RR^n$. Moreover, \cite[Lemma 8.4.6]{jarnicki1993} states that each $h\in\HH^1_+$ is of the form $h(\lambda)=c(\lambda-d)(1-\bar d\lambda)$ for some $c\in[0,\infty)$ and $d\in\CDD$. In that case $\BLHL{}=c|\lambda-d|^2$ for $\lambda\in\TT$.

\begin{obs}\label{obs_boundary_measure_of_one_dimensional_psi_z}If $\varphi\in\MMM^k$ has the boundary measure $\mu$ and $h\in\HH^k$, then for each $z\in\CC^k$ the function $\psi_z$ belongs to $\MMM^1$ and its boundary measure is equal to
$$
\BLHL{}\bullet(\RE z\,d\LEBT(\lambda)-d\mu(\lambda)).
$$
\end{obs}

\begin{proof}Repeat the argument employed in the proof of \cite[Lemma 3.7]{zajac20151337}.
\end{proof}

\begin{lem}\label{lem_D_fi_h_postaci_i_miary_graniczne}Let $D\in\mathcal{A}_d^n$, $\varphi, h\in\OO(\DD,\CC^n)$ be such that $\varphi(\DD)\subset\overline{D}$ and
$$
\RE\psi_z(\lambda)\leq 0,\quad\lambda\in\DD,\ z\in D.
$$
Let
$$
\RE\varphi^*\,d\LEBT+(0,\varrho)\,d\nu
$$
be the boundary measure of $\varphi$ written in the same form as in Lemma~\ref{lem_D_fi_miary_graniczne}. Then:
\begin{enumerate}[(i)]
\item\label{lem_D_fi_h_postaci_i_miary_graniczne_h_in_Hnd} $h\in\HH^n_d$,
\item\label{lem_D_fi_h_postaci_i_miary_graniczne_miara_graniczna_psi_z} for every $z\in\CC^n$ the boundary measure of $\psi_z$ is absolutely continuous with respect to $\LEBT$ and equal to
    $$
    \RE\left[\BLHSL{}\bullet(z-\varphi^*(\lambda))\right]\,d\LEBT(\lambda),
    $$
\item\label{lem_D_fi_h_postaci_i_miary_graniczne_blhl_bullet_ro_zerowe} $\BLHL{n-d+1,\ldots,n}\bullet\varrho(\lambda)=0$ for $\nu$-a.e.~$\lambda\in\TT$.
\end{enumerate}
\end{lem}

\begin{proof}Write $h=(h_1,\ldots,h_n)$. In view of Observation~\ref{obs_h_of_class_H1_if_fi_geodesic}, the map $h$ belongs to $H^1(\DD,\CC^n)$. Fix $z\in D$ and $\lambda\in\TT$ such that the radial limits of $h$ and $\varphi$ exist at $\lambda$. Take $j\in\lbrace n-d+1,\ldots,n\rbrace$. For every $t\in\RR$ the point $z+ite_j$ lies in $D$. Thus, from the assumptions and the equality \eqref{eq_g_granice_radialne_re_psi_z} we conclude that the function $t\mapsto\RE\left[\BLHSL{}\bullet ite_j\right]$ is bounded from above on $\RR$. This yields that $\BLHSL{j}\in\RR$, so $h_j\in\HH^1$, as required in \eqref{lem_D_fi_h_postaci_i_miary_graniczne_h_in_Hnd}.

To get \eqref{lem_D_fi_h_postaci_i_miary_graniczne_miara_graniczna_psi_z}, it suffices to prove that the boundary measure of $\psi_{\varphi(0)}$ is absolutely continuous with respect to $\LEBT$. Fix $r\in(0,1)$. By the assumptions, for all $z\in D$ and $\LEBT$-a.e.~$\lambda\in\TT$ (with the 'a.e.' being independent on $z$) we have $\RE\psi_z^*(\lambda)\leq 0$. In particular, applying this for $z=\varphi(r\lambda)$ we obtain
$$
\RE\psi_{\varphi(0)}^*(\lambda)=\RE\left[\BLHSL{}\bullet(\varphi(0)-\varphi^*(\lambda))\right]\leq -\RE\left[\BLHSL{}\bullet(\varphi(r\lambda)-\varphi(0))\right].
$$
The function $\lambda\mapsto h(\lambda)\bullet\lambda^{-1}(\varphi(r\lambda)-\varphi(0))$ is of class $H^1$, so integrating both sides of the above inequality we get
$$
\frac{1}{2\pi}\int_{\TT}\RE\psi_{\varphi(0)}^*(\lambda)\,d\LEBT(\lambda)\leq r\RE\left[-h(0)\bullet\varphi'(0)\right]=r\RE\psi_{\varphi(0)}(0).
$$
Thus
$$
\frac{1}{2\pi}\int_{\TT}\RE\psi_{\varphi(0)}^*(\lambda)\,d\LEBT(\lambda)\leq\RE\psi_{\varphi(0)}(0),
$$
because $r$ was chosen arbitrarily. Let $\omega_s$ be the singular part of the boundary measure of $\psi_{\varphi(0)}$ in its Lebesgue-Radon-Nikodym decomposition with respect to $\LEBT$. Since $\RE\psi_{\varphi(0)}\leq 0$ on $\DD$, the measure $\omega_s$ is negative. But the boundary measure of $\psi_{\varphi(0)}$ is equal to $\RE\psi_{\varphi(0)}^*\,d\LEBT+\omega_s$, so
$$
\RE\psi_{\varphi(0)}(0) = \frac{1}{2\pi}\int_{\TT}\RE\psi_{\varphi(0)}^*(\lambda)\,d\LEBT(\lambda)+\frac{1}{2\pi}\omega_s(\TT) \leq \RE\psi_{\varphi(0)}(0)+\frac{1}{2\pi}\omega_s(\TT).
$$
This inequality leads to the conclusion that $\omega_s(\TT)=0$ and completes the proof of \eqref{lem_D_fi_h_postaci_i_miary_graniczne_miara_graniczna_psi_z}.

It remains to prove \eqref{lem_D_fi_h_postaci_i_miary_graniczne_blhl_bullet_ro_zerowe}. Fix $z\in D$ and observe that
\begin{equation}\label{eq_ldfhpimg_psi_z_eq_psi_z_fi_a_plus_psi_0_fi_s}
\psi_z=\psi_z^{\varphi^a}+\psi_0^{\varphi^s}.
\end{equation}
Since $\varrho\,d\nu$ is the boundary measure of $\varphi_{n-d+1,\ldots,n}^s$, the condition \eqref{lem_D_fi_h_postaci_i_miary_graniczne_h_in_Hnd} and Observation~\ref{obs_boundary_measure_of_one_dimensional_psi_z} yield that the boundary measure of $\psi_0^{\varphi^s}$ is equal to
$$
-\BLHL{n-d+1,\ldots,n}\bullet\varrho(\lambda)\,d\nu(\lambda).
$$
We claim that it is positive. Indeed, for $\nu$-a.e.~$\lambda\in\TT$ we have $\varrho(\lambda)\in S_D$, by the choice of $\varrho$, and for $\LEBT$-a.e.~$\zeta\in\TT$ we have $\bar\zeta h^*(\zeta)\in W_D$, by \eqref{eq_g_granice_radialne_re_psi_z}. This means that
$$
\bar\zeta h_{n-d+1,\ldots,n}(\zeta)\bullet\varrho(\lambda) = \bar\zeta h^*(\zeta)\bullet (0,\varrho(\lambda))\leq 0.
$$
Now, in view of continuity of $h_{n-d+1,\ldots,n}$, fixing $\lambda$ and passing with $\zeta$ to $\lambda$, we obtain the desired inequality. A consequence of this and the equality \eqref{eq_ldfhpimg_psi_z_eq_psi_z_fi_a_plus_psi_0_fi_s} is that the boundary measure of $\psi_z^{\varphi^a}$ is negative, so $\varphi^a$ and $h$ staitsfy the assumptions of this lemma. Hence, they also satisfy the already proved condition \eqref{lem_D_fi_h_postaci_i_miary_graniczne_miara_graniczna_psi_z}. But the equalities
$$
\varphi_{1,\ldots,n-d}^*=(\varphi_{1,\ldots,n-d}^a)^*,\quad\RE\varphi_{n-d+1,\ldots,n}^*=\RE(\varphi_{n-d+1,\ldots,n}^a)^*,
$$
being valid $\LEBT$-a.e. on $\TT$, yield that the boundary measures of $\psi_z$ and $\psi_z^{\varphi_a}$ are equal. Therefore, from \eqref{eq_ldfhpimg_psi_z_eq_psi_z_fi_a_plus_psi_0_fi_s} it follows that $\RE\psi_0^{\varphi^s}\equiv 0$, what gives \eqref{lem_D_fi_h_postaci_i_miary_graniczne_blhl_bullet_ro_zerowe} and completes the proof.
\end{proof}

For a domain $D\in\mathcal{A}_d^n$ and a vector $v\in\CC^{n-d}\times\RR^d$ introduce the set
$$
P_D(v):=\lbrace p\in(\CC^{n-d}\times\RR^d)\cap\overline{D}:\RE((z-p)\bullet v)<0\text{ for all }z\in D\rbrace.
$$
It is clear that $P_D(v)$ is a closed and convex subset of $\partial D$. Moreover, if $P_D(v)\neq\varnothing$, then $v\in W_D$. In the case when the image of $D$ under the orthogonal projection on $\CC^{n-d}\times\RR^d$ is strictly convex in the geometric sense, each $P_D(v)$ has at most one element. The sets $P_D(v)$ represent certain geometric properties of $D$ which will be found useful in finding the part $\varphi^a$ of a complex geodesic.

\begin{thm}\label{th_warunek_konieczny_na_geodezyjna_w_Adn}Let $D\in\mathcal{A}_d^n$ and let $\varphi\in\OO(\DD,D)$ be a map with the boundary measure of the form
$$
\RE\varphi^*\,d\LEBT+(0,\varrho)\,d\nu,
$$
as in Lemma~\ref{lem_D_fi_miary_graniczne}. Assume that $\varphi$ is a complex geodesic for $D$ and take a map $h=(h_1,\ldots,h_n)\in\HH^n_d$ so that $\RE\psi_{\varphi(0)}(0)\neq 0$ and
$$
\RE\psi_z(\lambda)\leq 0,\quad\lambda\in\DD,\ z\in D.
$$
Then:
\begin{enumerate}[(i)]
\item\label{th_warunek_konieczny_na_geodezyjna_w_Adn_granice_radialne} for $\LEBT$-a.e.~$\lambda\in\TT$ one has that
    \begin{equation*}
    (\varphi_{1,\ldots,n-d}^*(\lambda),\RE\varphi_{n-d+1,\ldots,n}^*(\lambda))\in P_D(\BLHSL{}),
    \end{equation*}
\item\label{th_warunek_konieczny_na_geodezyjna_w_Adn_czesc_singularna} the measure $$\BLHL{n-d+1,\ldots,n}\bullet\varrho(\lambda)\,d\nu(\lambda)$$ is null.
\end{enumerate}

\noindent
Moreover, let $\nu'$ be a finite positive Borel measure on $\TT$, singular to $\LEBT$, and let $\varrho':\TT\longrightarrow\partial\BB_{\RR^d}$ be a Borel-measurable map such that $\varrho'(\lambda)\in S_D$ for $\nu'$-a.e.~$\lambda\in\TT$ and the measure
$$
\BLHL{n-d+1,\ldots,n}\bullet\varrho'(\lambda)\,d\nu'(\lambda)
$$
is null. If $\tau\in\MMM^n$ is such that $\tau^a\equiv\varphi^a$ and $\tau^s$ has the boundary measure $(0,\varrho')\,d\nu'$, then either $\tau(\DD)\subset\partial D$ or $\tau(\DD)\subset D$ and $\tau$ is a complex geodesic for $D$.
\end{thm}

\begin{rem}Given a domain $D\in\mathcal{A}_d^n$, a complex geodesic $\varphi$ for $D$ and a map $h\in\HH^n_d$ satisfying the assumptions of Theorem~\ref{th_warunek_konieczny_na_geodezyjna_w_Adn}, one can find its part $\varphi^a$ employing the condition \eqref{th_warunek_konieczny_na_geodezyjna_w_Adn_granice_radialne} together with the equalities \eqref{eq_poisson_formula_fi_a_1_nd} and \eqref{eq_poisson_formula_re_fi_a_nd1_n}. Especially when the image of $D$ under the orthogonal projection on $\CC^{n-d}\times\RR^d$ is strictly convex in the geometric sense, the map $\varphi^a$ is uniquely determined by $h$ up to an additive constant from $\lbrace 0\rbrace^{n-d}\times(i\RR)^d$. As for the map $\varphi^s$, the last part of the conclusion yields that we can hardly say more about it than it is stated in the condition \eqref{th_warunek_konieczny_na_geodezyjna_w_Adn_czesc_singularna}.
\end{rem}

\begin{proof}[Proof of Theorem~\ref{th_warunek_konieczny_na_geodezyjna_w_Adn}]
The condition \eqref{th_warunek_konieczny_na_geodezyjna_w_Adn_czesc_singularna} is a direct consequence of Lemma~\ref{lem_D_fi_h_postaci_i_miary_graniczne}. From \eqref{eq_g_granice_radialne_re_psi_z} it follows that for $\LEBT$-a.e.~$\lambda\in\TT$ there holds the inequality
$$
\RE\left[\BLHSL{}\bullet(z-\varphi^*(\lambda))\right]\leq 0,\quad z\in D.
$$
By the assumptions we have $h\not\equiv 0$, so the above mapping of the variable $z$ is open for $\LEBT$-a.e.~$\lambda\in\TT$. This means that the weak inequality can be in fact replaced by the strong one, what, together with Lemma~\ref{lem_D_fi_h_postaci_i_miary_graniczne}~\eqref{lem_D_fi_h_postaci_i_miary_graniczne_h_in_Hnd}, gives the condition \eqref{th_warunek_konieczny_na_geodezyjna_w_Adn_granice_radialne}.

To prove the remaining part of the conclusion, take $\nu'$, $\varrho'$ and $\tau$ as in the assumptions. Since $\tau_{1,\ldots,n-d}^s\equiv 0\equiv \varphi_{1,\ldots,n-d}^s$, one has that
$$
\tau_{1,\ldots,n-d}\equiv\tau_{1,\ldots,n-d}^a\equiv\varphi_{1,\ldots,n-d}^a\equiv\varphi_{1,\ldots,n-d},
$$
so $\tau_{1,\ldots,n-d}\in H^1(\DD,\CC^{n-d})$. For $\LEBT$-a.e.~$\lambda\in\TT$ it holds that
$$
\RE(\tau^s)^*(\lambda)=0=\RE(\varphi^s)^*(\lambda),$$
what gives
$$
\left(\tau_{1,\ldots,n-d}^*(\lambda),\RE\tau_{n-d+1,\ldots,n}^*(\lambda)\right)=\left(\varphi_{1,\ldots,n-d}^*(\lambda),\RE\varphi_{n-d+1,\ldots,n}^*(\lambda)\right).
$$
Thus, Lemma~\ref{lem_D_fi_miary_graniczne} yields that $\tau(\DD)\subset\overline{D}$. If $\tau(\DD)\subset\partial D$, then we are done. In the opposite case we have $\tau(\DD)\subset D$ and, by Observation~\ref{obs_boundary_measure_of_one_dimensional_psi_z},
$$
\RE\psi_0^{\tau^s}\equiv\RE\psi_0^{\tau_{n-d+1,\ldots,n}^s, h_{n-d+1,\ldots,n}}\equiv 0
$$
and, similarly, $\RE\psi_0^{\varphi^s}\equiv 0$. Therefore, for every $z\in\CC^n$ one has that
$$
\RE\psi_z^{\tau}\equiv\RE\psi_z^{\tau^a}+\RE\psi_0^{\tau^s}\equiv\RE\psi_z^{\tau^a}\equiv\RE\psi_z^{\varphi^a}\equiv\RE\psi_z^{\varphi^a}+\RE\psi_0^{\varphi^s}\equiv\RE\psi_z.
$$
In particular, $\RE\psi_z^{\tau}(\lambda)\leq 0$ for all $\lambda\in\DD$ and $z\in D$. In view of Lemma~\ref{lem_g_ogolne_warunki_wystarczajace}, to complete the proof it suffices to show that $\RE\psi_{\tau(0)}(0)\neq 0$. But the map $z\mapsto\RE\psi_z(0)$ is either open or identically equal to $\RE\psi_{\varphi(0)}(0)$. Since it takes only non-positive values, in both cases its image has to lie in the interval $(-\infty,0)$.
\end{proof}

The following fact, under certain additional assumption, stands for an inverse of Theorem~\ref{th_warunek_konieczny_na_geodezyjna_w_Adn}.

\begin{thm}\label{th_warunek_wystarczajacy_na_geodezyjna_w_Adn}Let $D\in\mathcal{A}_d^n$ be a domain having bounded image under the projection on first $n-d$ coordinates and let $\varphi\in\OO(\DD,D)$ be a map with the boundary measure of the form
$$
\RE\varphi^*\,d\LEBT+(0,\varrho)\,d\nu,
$$
as in Lemma~\ref{lem_D_fi_miary_graniczne}. If there exists a map $h\in\HH^n_d$ satisfying the conditions \eqref{th_warunek_konieczny_na_geodezyjna_w_Adn_granice_radialne} and \eqref{th_warunek_konieczny_na_geodezyjna_w_Adn_czesc_singularna} from Theorem~\ref{th_warunek_konieczny_na_geodezyjna_w_Adn}, then $\varphi$ is a complex geodesic for $D$.
\end{thm}

The assumption on $D$ from Theorem~\ref{th_warunek_wystarczajacy_na_geodezyjna_w_Adn} is fulfilled for example when $d=n$ (then $D$ is a convex tube domain; see \cite[Theorem 3.1]{zajac20161865}) and when $D$ is a bounded convex domain (cf.~\cite[Subsection 8.2]{jarnicki1993}). It is interesting whether the conclusion of the theorem is valid for an arbitrary $D\in\mathcal{A}_d^n$.

\begin{proof}Write $\varphi=(\varphi_1,\ldots,\varphi_n)$ and $h=(h_1,\ldots,h_n)$. One has that
$$
\psi_z^{\varphi,h}(\lambda) =\psi_{z_1}^{\varphi_1,h_1}(\lambda)+\ldots+\psi_{z_{n-d}}^{\varphi_{n-d},h_{n-d}}(\lambda)+\psi_{z_{n-d+1,\ldots,n}}^{\varphi_{n-d+1,\ldots,n},h_{n-d+1,\ldots,n}}(\lambda)
$$
for $\lambda\in\DD$ and $z=(z_1,\ldots,z_n)\in\CC^n$. If $j\in\lbrace 1,\ldots,n-d\rbrace$, then $\varphi_j$ is bounded, so $\psi_{z_j}^{\varphi_j,h_j}$ is of the class $H^1$. Moreover, by the assumption \eqref{th_warunek_konieczny_na_geodezyjna_w_Adn_czesc_singularna} and Observation~\ref{obs_boundary_measure_of_one_dimensional_psi_z}, the boundary measure of $\psi_{z_{n-d+1,\ldots,n}}^{\varphi_{n-d+1,\ldots,n},h_{n-d+1,\ldots,n}}$ equals to
$$
\RE\left[\BLHL{n-d+1,\ldots,n}\bullet(z_{n-d+1,\ldots,n}-\varphi_{n-d+1,\ldots,n}^*(\lambda))\right]\,d\LEBT(\lambda).
$$
From these considerations it follows that $\psi_z^{\varphi,h}\in\MMM^1$ and its boundary measure is equal to
$$
\RE\left[\BLHSL{}\bullet(z-\varphi^*(\lambda))\right]\,d\LEBT(\lambda).
$$
In view of the assumption \eqref{th_warunek_konieczny_na_geodezyjna_w_Adn_granice_radialne}, this measure is negative when $z\in D$, what leads to the conclusion that $\RE\psi_z^{\varphi,h}(\lambda)\leq 0$ for all $\lambda\in\DD$ and $z\in D$. Finally, $\RE\psi_{\varphi(0)}^{\varphi,h}(0)\neq 0$, because otherwise the maximum principle yields that $\RE\psi_{\varphi(0)}^{\varphi,h}\equiv 0$, what contradicts the assumption \eqref{th_warunek_konieczny_na_geodezyjna_w_Adn_granice_radialne}. Now Lemma~\ref{lem_g_ogolne_warunki_wystarczajace} does the job.
\end{proof}

\begin{ex}Consider the domain
$$
D:=\lbrace (z_1,z_2)\in\CC^2:|z_1|^2+|\RE z_2|^2<1\rbrace.
$$
It belongs to the family $\mathcal{A}_1^2$. One can check that $W_D=\CC\times\RR$, $S_D=\lbrace 0\rbrace$, and
$$
P_D(v)=\left\lbrace\frac{\bar v}{\|v\|}\right\rbrace,\;v\in(\CC\times\RR)_*.
$$

Let $\varphi=(\varphi_1,\varphi_2)\in\OO(\DD,D)$ be a complex geodesic for $D$. Take a map $h=(h_1,h_2)\in\HH^2_1$ satisfying the assumptions of Theorem~\ref{th_warunek_konieczny_na_geodezyjna_w_Adn}. According to Lemma~\ref{lem_D_fi_miary_graniczne}, we have $\varphi^s\equiv 0$ and $\varphi^a\equiv\varphi$. Therefore, from Theorem~\ref{th_warunek_konieczny_na_geodezyjna_w_Adn} we conclude that
$$
\left(\varphi_1^*(\lambda),\RE\varphi_2^*(\lambda)\right) = \left(\LBHSL{1},\BLHL{2}\right)\cdot\left\|\Bigl(\BLHSL{1},\BLHL{2}\Bigr)\right\|^{-1}
$$
for $\LEBT$-a.e.~$\lambda\in\TT$. Now, one can recover $\varphi$, up to a constant $b\in\lbrace 0\rbrace\times(i\RR)$, from the equalities \eqref{eq_poisson_formula_fi_a_1_nd} and \eqref{eq_poisson_formula_re_fi_a_nd1_n}.

On the other hand, fix a mapping $h=(h_1,h_2)\in\HH^2_1$ which does not vanish identically and define, for $\lambda\in\DD$,
\begin{align*}
\varphi_1(\lambda)&:=\frac{1}{2\pi}\int_{\TT}\frac{1-|\lambda|^2}{|\zeta-\lambda|^2}\cdot\zeta\overline{h_1^*(\zeta)}\cdot\left\|\Bigl(\bar\zeta h_1^*(\zeta),\bar\zeta h_2(\zeta)\Bigr)\right\|^{-1} d\LEBT(\zeta),\\
\varphi_2(\lambda)&:=\frac{1}{2\pi}\int_{\TT}\frac{\zeta+\lambda}{\zeta-\lambda}\cdot\bar\zeta h_2(\zeta)\cdot\left\|\Bigl(\bar\zeta h_1^*(\zeta),\bar\zeta h_2(\zeta)\Bigr)\right\|^{-1} d\LEBT(\zeta).
\end{align*}
Assume, additionally, that $\varphi(0)\in D$, where $\varphi=(\varphi_1,\varphi_2)$. Then Theorem~\ref{th_warunek_wystarczajacy_na_geodezyjna_w_Adn} applied to $\varphi$ and $h$ guarantees that $\varphi$ is a complex geodesic for $D$, but only under the additional assumption that the function $\varphi_1$ is holomorphic. This can, however, fail, as it is shown by the quite simple example of $h(\zeta):=(\zeta^2,0)$, because then $\varphi_1(\lambda)=\bar\lambda$.

In general situation, the question on holomorphicity of the first $n-d$ components of a map $\varphi$ obtained in such a way, strongly depends on the geometry of $D$. It is worthy to point out that in tube domains (that is, those from $\mathcal{A}_n^n$), considered in \cite{zajac20151337} and \cite{zajac20161865}, this problem did not arise, because there everything about $\varphi$ was expressed in terms of its real part and the entire $\varphi$ was defined in a similar way as $\varphi_2$ in our example.
\end{ex}

\begin{ex}Let
$$
D:=\lbrace(z_1,z_2)\in\CC^2:\RE z_2>|z_1|^2\rbrace.
$$
As previously, this domain belongs to the family $\mathcal{A}_1^2$. We have
$$
W_D=(\CC\times(-\infty,0))\cup\lbrace (0,0)\rbrace\text{ and }S_D=[0,\infty).
$$
Moreover,
$$
P_D(v)=\left\lbrace\left(-\frac{\bar v_1}{2v_2},\left|\frac{v_1}{2v_2}\right|^2\right)\right\rbrace,
$$
when $v=(v_1,v_2)\in \CC\times(-\infty,0)$, and $P_D(v)=\varnothing$ otherwise.

Let $\varphi=(\varphi_1,\varphi_2)\in\OO(\DD,D)$ be a complex geodesic for $D$ with the boundary measure written as in Lemma~\ref{lem_D_fi_miary_graniczne}. Take $h=(h_1,h_2)\in\HH^2_1$ as in Theorem~\ref{th_warunek_konieczny_na_geodezyjna_w_Adn}. From the condition \eqref{th_warunek_konieczny_na_geodezyjna_w_Adn_granice_radialne} it follows that $\BLHSL{}\in W_D$ for $\LEBT$-a.e.~$\lambda\in\TT$. Thus $h_2\in-\HH^1_+$ and $h_2\not\equiv 0$, because $h\not\equiv 0$. In particular, $h_2$ has at most one zero on $\TT$ (counting without multiplicities).
By Theorem~\ref{th_warunek_konieczny_na_geodezyjna_w_Adn} \eqref{th_warunek_konieczny_na_geodezyjna_w_Adn_czesc_singularna}, $\BLHL{2}\varrho(\lambda)=0$ for $\nu$-a.e.~$\lambda\in\TT$, so
$$
\varrho\,d\nu=\alpha\delta_{\lambda_0}
$$
for some $\alpha\geq 0$ and $\lambda_0\in\TT$ such that $\alpha h_2(\lambda_0)=0$. This gives that
$$
\varphi_2^s(\lambda)=\frac{\alpha}{2\pi}\frac{\lambda_0+\lambda}{\lambda_0-\lambda},\quad\lambda\in\DD.
$$
Moreover, it is clear that $\varphi_1^s\equiv 0$. As for the map $\varphi^a$, for $\LEBT$-a.e.~$\lambda\in\TT$ one has that
$$
\left(\varphi_1^*(\lambda),\RE\varphi_2^*(\lambda)\right)=\left(-\frac{\LBHSL{1}}{\Bigl. 2\BLHL{2}\Bigr.},\left|\frac{\BLHSL{1}}{\Bigl. 2\BLHL{2}\Bigr.}\right|^2\right),
$$
by Theorem~\ref{th_warunek_konieczny_na_geodezyjna_w_Adn} \eqref{th_warunek_konieczny_na_geodezyjna_w_Adn_granice_radialne}. Now, we can recover $\varphi^a$ employing the equalities \eqref{eq_poisson_formula_fi_a_1_nd} and \eqref{eq_poisson_formula_re_fi_a_nd1_n}. Finally, having $\varphi^a$ and $\varphi^s$ calculated, we are able to derive a formula for $\varphi$, an arbitrarily chosen complex geodesic for $D$.
\end{ex}

\section{$\mathbb C$-convexity of semitube domains}\label{sect:ccstd}

The aim of this section is to prove the following result.

\begin{thm}\label{thm:ccstd}Let $D$ be a domain in $\mathbb R^{2n-1}$ such that $\{x\in\mathbb R^{2n-1}:x'=a'\}\not\subset\partial D$ for any $a\in\partial D$. Then $\mathcal S_D$ is $\mathbb C$-convex if and only if it is convex.
\end{thm}

\begin{rem}\label{rem:example}(a) If $D=\Omega\times\mathbb R$ for some domain $\Omega\subset\mathbb R^{2n-2}$ (i.e.~$\{x\in\mathbb R^{2n-1}:x'=a'\}\subset\partial D$ for any $a\in\partial D$) then the assertion of Theorem~\ref{thm:ccstd} is no longer true. Indeed, if $\Omega$ is non-convex and $\mathbb C$-convex (as a domain considered in $\mathbb C^{n-1}$), then $\mathcal S_D$ is non-convex but $\mathbb C$-convex semitube domain.

(b) Although the condition imposed onto the domain $D$ in Theorem~\ref{thm:ccstd} seems to be a technical one, the example in part (a) shows that some restriction of this kind is needed, if we want to have the equivalence of the notions of convexity and $\mathbb C$-convexity in the class of semitube domains. It is an open question whether the condition assumed in Theorem~\ref{thm:ccstd} is a necessary one for the aforementioned equivalence.
\end{rem}

In what follows we shall need the notion of linear convexity. Recall that a domain $D\subset\mathbb C^n$ is called (cf.~\cite{hormander1994}, \cite{andersson2004}) \emph{linearly convex}, if its complement is a union of affine complex hyperplanes. Note that any $\mathbb C$-convex domain is linearly convex (cf.~\cite[Theorem 2.3.9]{andersson2004}), but the converse is not true (cf.~\cite[p.~26]{andersson2004} or \cite[Theorem 1~(ii)]{nikolov2008149}).

We begin with the following simple observation which is crucial in the proof of Theorem~\ref{thm:ccstd}.

\begin{prop}\label{prop:lcstd}Let $D$ be a domain in $\mathbb R^{2n-1}$, $n>1$. Then the following conditions are equivalent:
\begin{enumerate}[(i)]
\item\label{item:lci}$\mathcal S_D$ is linearly convex,
\item\label{item:lcii}for any $a=(a',a_{2n-1})\in\mathbb R^{2n-1}\setminus D$ there exists affine subspace $H\subset\mathbb R^{2n-1}$, $\codim_{\mathbb R}H\in\{1,2\}$, such that $a\in H$, $H\cap D=\varnothing$,
\begin{equation}\label{eq:H21}
H=\begin{cases}\{x\in\mathbb R^{2n-1}:b\bullet(x'-a')=\tilde b\bullet(x'-a')=0\},\quad&\textnormal{if }\codim_{\mathbb R}H=2\\\{x\in\mathbb R^{2n-1}:x_{2n-1}=a_{2n-1}-b\bullet(x'-a')\},\quad&\textnormal{if }\codim_{\mathbb R}H=1\end{cases}
\end{equation}
for some $b\in\mathbb R^{2n-2}$, where $\tilde b=(\tilde b_1,\dots,\tilde b_{2n-2})$,
\begin{equation*}\label{eq:tildeb}
\tilde b_j=\begin{cases}-b_{j+1},\quad&\textnormal{if }j\textnormal{ is odd}\\b_{j-1},\quad&\textnormal{if }j\textnormal{ is even}\end{cases},\quad j=1,2,\dots,2n-2;
\end{equation*}
moreover, if $\codim_{\mathbb R}H=2$, then $b\neq0$.
\end{enumerate}
\end{prop}

Let $\iota:\mathbb R^{2n-1}\longrightarrow\mathbb C^n$ be defined by
\begin{equation*}
\iota(x_1,x_2,\dots,x_{2n-1}):=(x_1+ix_2,\dots,x_{2n-3}+ix_{2n-2},x_{2n-1}).
\end{equation*}
Note that $\Pi\circ\iota$ is the identity of $\mathbb R^{2n-1}$.

\begin{proof}[Proof of Proposition~\ref{prop:lcstd}]\textit{(\ref{item:lci})}$\Longrightarrow$\textit{(\ref{item:lcii})}. Fix $a\in\mathbb R^{2n-1}\setminus D$. Since $\mathcal S_D$ is linearly convex, there exists an affine complex hyperplane $L\subset\mathbb C^n$ such that $\iota(a)\in L$ and $L\cap\mathcal S_D=\varnothing$. Observe that
\begin{equation*}
L=\{z\in\mathbb C^n:\alpha\bullet(z-\iota(a))=0\}
\end{equation*}
for some $\alpha=(\alpha',\alpha_n)\in(\mathbb C^n)_*$.

If $\alpha_n=0$ then $\alpha'\neq0$ and $z\in L$ if and only if $\alpha'\bullet(z'-\iota(a)')=0$, i.e.~$H:=\Pi(L)$ is of the form (\ref{eq:H21}), $\codim_{\mathbb R}H=2$, with
\begin{equation}\label{eq:b}
b_j:=\begin{cases}\re\alpha_{(j+1)/2},\quad&\textnormal{if }j\textnormal{ is odd}\\-\im\alpha_{j/2},\quad&\textnormal{if }j\textnormal{ is even}\end{cases},\quad j=1,2,\dots,2n-2.
\end{equation}

If $\alpha_n\neq0$ then without loss of generality we may assume that $\alpha_n=1$. Hence $z\in L$ if and only if
\begin{align*}
\re z_n&=a_{2n-1}-\re\left(\alpha'\bullet(z'-\iota(a)')\right),\\
\im z_n&=-\im\left(\alpha'\bullet(z'-\iota(a)')\right),
\end{align*}
i.e.~$H:=\Pi(L)$ is of the form (\ref{eq:H21}), $\codim_{\mathbb R}H=1$, with $b=(b_1,b_2,\dots,b_{2n-2})$ defined by (\ref{eq:b}).

\textit{(\ref{item:lcii})}$\Longrightarrow$\textit{(\ref{item:lci})}. Take arbitrary $w\in\mathbb C^n\setminus\mathcal S_D$. Let $H$ be as in (ii) for $a:=\Pi(w)\notin D$.

If $\codim_{\mathbb R}H=2$ then put
\begin{equation*}
L:=\Pi^{-1}(H)=\{z\in\mathbb C^n:b\bullet(\Pi(z)'-a')=\tilde b\bullet(\Pi(z)'-a')=0\}.
\end{equation*}
Note that $z\in L$ if and only if
\begin{equation*}
0=b\bullet(\Pi(z)'-a')+i(\tilde b\bullet(\Pi(z)'-a'))=\sum_{j=1}^{n-1}(b_{2j-1}-ib_{2j})(z_j-\iota_j(a)),
\end{equation*}
i.e.~$L=\{z\in\mathbb C^n:\alpha\bullet(z-\iota(a))=0\}$ with $\alpha=(\alpha_1,\dots,\alpha_n)\in(\mathbb C^n)_*$, where
\begin{equation}\label{eq:alpha'}
\alpha_j:=b_{2j-1}-ib_{2j},\quad j=1,2,\dots,n-1,\qquad\alpha_n:=0,
\end{equation}
is affine complex hyperplane in $\mathbb C^n$ with $w\in L$ and $L\cap\mathcal S_D=\varnothing$.

If $\codim_{\mathbb R}H=1$ then observe that
\begin{equation*}
\Pi^{-1}(H)=\{z\in\mathbb C^n:\re z_n=a_{2n-1}-b\bullet(\Pi(z)'-a')\}.
\end{equation*}
Put
\begin{equation*}
L:=\{z\in\Pi^{-1}(H):\im z_n=-\tilde b\bullet(\Pi(z)'-a')\}.
\end{equation*}
Note that $z\in L$ if and only if
\begin{align*}
0&=b\bullet(\Pi(z)'-a')+\re z_n-a_{2n-1}+i(\tilde b\bullet(\Pi(z)'-a')+\im z_n)\\{}&=\sum_{j=1}^{n-1}(b_{2j-1}-ib_{2j})(z_j-\iota_j(a))+z_n-\iota_n(a),
\end{align*}
i.e.~$L=\{z\in\mathbb C^n:\alpha\bullet(z-\iota(a))=0\}$ with $\alpha=(\alpha_1,\dots,\alpha_n)\in(\mathbb C^n)_*$, where
\begin{equation*}
\alpha_j:=b_{2j-1}-ib_{2j},\quad j=1,2,\dots,n-1,\qquad\alpha_n:=1,
\end{equation*}
is affine complex hyperplane in $\mathbb C^n$ with $w\in L$ and $L\cap\mathcal S_D=\varnothing$.
\end{proof}

For a domain $G\subset\mathbb C^n$ and a point $w\in\mathbb C^n$, we denote by $\Gamma_G(w)$ the set of all complex hyperplanes $L$ such that $(w+L)\cap G=\varnothing$. One may identify this set with a subset of complex projective space $\mathbb P^{n-1}$: here $L=\{z\in\mathbb C^n:b\bullet z=0\}$ is identified with $[b]\in\mathbb P^{n-1}$. In the proof of Theorem~\ref{thm:ccstd} we shall use the following characterization of $\mathbb C$-convexity: if a domain $G\subset\mathbb C^n$, $n>1$, is $\mathbb C$-convex then for any $w\in\partial G$ the set $\Gamma_G(w)$ is non-empty and connected (cf.~\cite[p.~46]{andersson2004}).

Let $0\leq d\leq k$ be two integers. The Grassmann manifold $\Gr(d,\mathbb R^k)$ is the set of all $d$-dimensional real subspaces of $\mathbb R^k$ which is topologized as a quotient space (see e.g.~\cite[p.~56]{milnor1974} for details). In what follows we shall use the following result.

\begin{lem}[\cite{milnor1974}, Lemma 5.1]\label{lem:compact}The Grassmann manifold $\Gr(d,\mathbb R^k)$ is compact.
\end{lem}

\begin{proof}[Proof of Theorem~\ref{thm:ccstd}]Assume $\mathcal S_D$ is $\mathbb C$-convex. It suffices to show that $D$ is convex. Suppose $D$ is not convex, i.e.~there is a point $a\in\partial D$ such that for any affine real hyperplane $P\subset\mathbb R^{2n-1}$ with $a\in P$ we have $P\cap D\neq\varnothing$. Since $\mathcal S_D$ is $\mathbb C$-convex, it is linearly convex. Consequently, by Proposition~\ref{prop:lcstd}, there is an affine real subspace $H$ of the form (\ref{eq:H21}), $\codim_{\mathbb R}(H)=2$, with $a\in H$ and $H\cap D=\varnothing$. In particular, $L:=\Pi^{-1}(H)$ is an affine complex hyperplane in $\mathbb C^n$ (see proof of Proposition~\ref{prop:lcstd}, part (ii)$\Longrightarrow$(i), case $\codim_{\mathbb R}H=2$) with $w\in L$ and $L\cap\mathcal S_D=\varnothing$, i.e.~$[(\alpha',0)]\in\Gamma_{\mathcal S_D}(w)$ for any $w\in\Pi^{-1}(a)$, where $\alpha'=(\alpha_1,\dots,\alpha_{n-1})\in(\mathbb C^{n-1})_*$ is defined via (\ref{eq:alpha'}).

Without loss of generality we may assume that
\begin{equation*}
t_0:=\sup\{t\geq0:[a,(a',a_{2n-1}+t)]\subset\partial D\}\in\mathbb R,
\end{equation*}
where $[x,y]:=\{\lambda y+(1-\lambda)x:\lambda\in[0,1]\}$ denotes the segment with endpoints $x$ and $y$. Let $\tilde a:=(a',a_{2n-1}+t_0)$. Observe that $\tilde a\in H\cap\partial D$. Set $\tilde w:=\iota(\tilde a)$ and note that $\tilde w\in\iota(\partial D)$ and $\tilde w\in\Pi^{-1}(\tilde a)$. We consider two cases.

\emph{Case 1}. There is $[\beta]\in\Gamma_{\mathcal S_D}(\tilde w)$ with $\beta_n\neq0$. Since $\Gamma_{\mathcal S_D}(\tilde w)$ is connected and $[(\alpha',0)]\in\Gamma_{\mathcal S_D}(\tilde w)$, there is a sequence $([\beta^k])_{k\in\mathbb N}\subset\Gamma_{\mathcal S_D}(\tilde w)$ with $\beta^k_n\neq0$, $k\in\mathbb N$, such that $\lim_{k\to\infty}[\beta^k]=[(\gamma',0)]$ for some $[(\gamma',0)]\in\Gamma_{\mathcal S_D}(\tilde w)$. Observe that
\begin{equation*}
L_k=\{z\in\mathbb C^n:\beta^k\bullet(z-\tilde w)=0\},\quad k\in\mathbb N,
\end{equation*}
is an affine complex hyperplane such that $\tilde w\in L_k$ and $L_k\cap\mathcal S_D=\varnothing$. In particular, $H_k:=\Pi(L_k)$ is an affine real subspace of $\mathbb R^{2n-1}$, $\codim_{\mathbb R}H_k=1$, such that $\tilde a\in H_k$, $H_k\cap D=\varnothing$, $k\in\mathbb N$ (see proof of Proposition~\ref{prop:lcstd}, part (i)$\Longrightarrow$(ii)). Since the Grassmann manifold $\Gr(2n-2,\mathbb R^{2n-1})$ is compact (cf.~Lemma~\ref{lem:compact}), we may assume that $\lim_{k\to\infty}H_k=\tilde H$, where $\tilde H$ is an affine real subspace of $\mathbb R^{2n-1}$, $\codim_{\mathbb R}\tilde H=1$, with $\tilde a\in\tilde H$, $\tilde H\cap D=\varnothing$. It remains to observe that $a\in\tilde H$ (indeed, the equality $\lim_{k\to\infty}\beta_n^k=0$ implies that the equation of $\tilde H$ does not depend on the last, $2n-1$st, variable). In other words, $\tilde H$ is an affine real hyperplane in $\mathbb R^{2n-1}$ passing through $a$ and disjoint from $D$, which contradicts the choice of $a$.

\emph{Case 2}. $\beta_n=0$ for any $[\beta]\in\Gamma_{\mathcal S_D}(\tilde w)$. Consequently, there is no affine real hyperplane $\tilde P\subset\mathbb R^{2n-1}$ with $\tilde a\in\tilde P$, $\tilde P\cap D=\varnothing$ such that $l:=\{x\in\mathbb R^{2n-1}:x'=a'\}\not\subset\tilde P$ (follow the proof of Proposition~\ref{prop:lcstd}, part (\ref{item:lcii})$\Longrightarrow$(\ref{item:lci}), case $\codim_{\mathbb R}H=1$, with $H$ replaced by $\tilde P$). Moreover, there is no affine real hyperplane $\tilde P\subset\mathbb R^{2n-1}$ with $l\subset\tilde P$, $\tilde a\in\tilde P$, and $\tilde P\cap D=\varnothing$ (such $\tilde P$ would also be a supporting hyperplane for the point $a$---a contradiction). Hence we conclude that there is no affine real hyperplane containing $\tilde a$ and disjoint from $D$.

In what follows we shall use the following lemma.

\begin{lem}\label{lem:t}For any neighborhood $U$ of $\tilde w$ there is a point $\hat w\in U\cap\iota(\partial D)$ such that there exists a $[\beta]\in\Gamma_{\mathcal S_D}(\hat w)$ with $\beta_n\neq0$.
\end{lem}

We postpone the proof of the Lemma~\ref{lem:t} and continue the proof of Theorem~\ref{thm:ccstd}. Consequently, there are
\begin{itemize}
  \item a sequence $(w^k)_{k\in\mathbb N}\subset\iota(\partial D)$, $\lim_{k\to\infty}w^k=\tilde w$,
  \item a sequence $([\beta^k])_{k\in\mathbb N}\subset\mathbb P^{n-1}$, $\beta_n^k\neq0$, $k\in\mathbb N$, with $[\beta^k]\in\Gamma_{\mathcal S_D}(w^k)$ such that $L_k\cap\mathcal S_D=\varnothing$, where
  \begin{equation*}
  L_k:=\{z\in\mathbb C^n:\beta^k\bullet(z-w^k)=0\},\quad k\in\mathbb N.
  \end{equation*}
\end{itemize}
In particular, $H_k:=\Pi(L_k)$ is an affine real subspace of $\mathbb R^{2n-1}$, $\codim_{\mathbb R}H_k=1$, such that $\Pi(w^k)\in H_k$, $H_k\cap D=\varnothing$, $k\in\mathbb N$. Again, using compactness argument as in Case 1, we may assume that $\lim_{k\to\infty}H_k=H_0$, where $H_0$ is an affine real subspace of $\mathbb R^{2n-1}$, $\codim_{\mathbb R}H_0=1$, with $\tilde a\in H_0$, $H_0\cap D=\varnothing$---a contradiction, since there is no affine real hyperplane containing $\tilde a$ and disjoint from $D$.
\end{proof}

\begin{proof}[Proof of Lemma~\ref{lem:t}]Fix a neighborhood $U$ of $\tilde w$ and an $\varepsilon>0$ such that $\mathbb B(\tilde w,\varepsilon)\subset U$. According to the definition of $\tilde w$, there exists a $0<\tilde t<\varepsilon$ such that
\begin{equation*}
\tilde w(\tilde t):=(\tilde w',\tilde w_n+\tilde t)\in\mathbb C^n\setminus\iota(\overline D).
\end{equation*}
Consequently, there is an $r>0$ such that $\mathbb B(\tilde w(\tilde t),r)\subset\mathbb C^n\setminus\iota(\overline D)$ and $\mathbb B(\tilde w(\tilde t),r)\subset\mathbb B(\tilde w,\varepsilon)$. On the other hand, since $\tilde w\in\iota(\partial D)$, there exists a $\zeta\in\mathbb B(\tilde w,r)\cap\iota(D)$. Set
\begin{equation*}
\hat t:=\sup\{t>0:[\zeta,(\zeta',\zeta_n+t)]\subset\iota(D)\}.
\end{equation*}
Note that $\hat t<\tilde t$. Then $\hat w:=(\zeta',\zeta_n+\hat t)\in\mathbb B(\tilde w,\varepsilon)\cap\iota(\partial D)$ will do the job. Indeed, since $\Pi(\hat w)\in\partial D$, $\Pi(\zeta)\in D$, and $(\Pi(\hat w))'=(\re\zeta_1,\im\zeta_1,\dots,\re\zeta_{n-1},\im\zeta_{n-1})=(\Pi(\zeta))'$, linear convexity of $\mathcal S_D$ and Proposition~\ref{prop:lcstd} imply that there exists an affine real hyperplane $\hat H\subset\mathbb R^{2n-1}$ such that $\Pi(\hat w)\in\hat H$ and $\hat H\cap D=\varnothing$. Moreover, it is of the form
\begin{equation*}
\hat H=\{x\in\mathbb R^{2n-1}:x_{2n-1}=(\Pi(\hat w))_{2n-1}-\hat b\bullet(x'-(\Pi(\hat w))')\}
\end{equation*}
for some $\hat b\in\mathbb R^{2n-2}$. Let $\hat L$ be defined as $L$ in the proof of Proposition~\ref{prop:lcstd}, part (\ref{item:lcii})$\Longrightarrow$(\ref{item:lci}), case $\codim_{\mathbb R}H=1$, with $H$ replaced by $\hat H$ and $a$ replaced by $\Pi(\hat w)$. Consequently, such an $\hat L$ is of the form
\begin{equation*}
\hat L=\{z\in\mathbb C^n:\beta\bullet(z-\hat w)=0\}
\end{equation*}
for some $\beta=(\beta_1,\dots,\beta_n)\in\mathbb C^n$ with $\beta_n=1$. Consequently, $\hat L$ is an affine complex hyperplane such that $\hat w\in\hat L$ and $\hat L\cap\mathcal S_D=\varnothing$, i.e.~$[\beta]\in\Gamma_{\mathcal S_D}(\hat w)$ with $\beta_n\neq0$.
\end{proof}

\bibliographystyle{amsalpha}
\bibliography{bib_pz}

\end{document}